\documentclass[a4paper, 10 pt]{amsart}
\usepackage{amssymb}
\usepackage{enumerate}
\usepackage[dvips]{graphicx} 
\usepackage{xspace}
\usepackage[applemac]{inputenc}
\usepackage[T1]{fontenc}
\usepackage{lmodern}
\usepackage{textcomp}
\theoremstyle{plain}
\newtheorem{thm}{Theorem}[section]

\newtheorem{lem}[thm]{Lemma}

\theoremstyle{definition}

\theoremstyle{remark}

\input{Definitions.tex}
\title{A converse to Schreier's index-rank formula}
\author{Ralph Strebel}
\date{\today}
\begin{document}
\thispagestyle{plain}
\maketitle 

\textbf{1. }
Let $G$ be a group and let $\varphi$ be a real-valued function, 
defined on all subgroups $G_1$ with finite index in $G$.
Call $\varphi$ \emph{submultiplicative} 
if the inequality
\begin{equation}
\label{eq:submultiplicative}
\varphi(G_2) \leq [G_1 : G_2] \cdot \varphi(G_1)
\end{equation}
holds for all pairs of subgroups $G_2 \leq G_1$ of $G$ 
with $G_2$ a subgroup of finite index in $G$;
call it \emph{multiplicative} if equality holds throughout in inequality \eqref{eq:submultiplicative}.

The rank $\rk(G)$ of a finitely generated group $G$ is, by definition, 
the minimal number of elements generating it.
By a Theorem of Reidemeister's, 
subgroups with finite index in finitely generated groups are finitely generated; 
moreover, the rank function on its subgroups of finite index is submultiplicative
(see, \eg \cite[p. 89, Thm.\;2.7]{MKS66}). 
For finitely generated \emph{free} groups, 
a sharper result holds: Schreier's index-rank  formula
(\cite{Sch27} or \cite[p.\;16, Prop.\;3.9]{LS77}) asserts 
that the function  $\rk-1$ is multiplicative.
It follows
that the function $\rk-1$ is submultiplicative on every finitely generated group.
\medskip

\textbf{2. }
In  \cite[§ 2]{LuDr81} Lubotzky and van den Dries ask 
whether a finitely generated group $G$ is necessarily free 
if its function $\rk-1$ is multiplicative
and if, in addition, $G$ is \emph{residually finite}.
One has, of course, to require 
that $G$ contains sufficiently many subgroups of finite index
in order to rule out counter-examples 
like finitely generated infinite simple groups or,
more concretely, G. Higman's 4-generator 4-relator group
\[
\langle 
a_1, a_2, a_3, a_4 \mid  a_i a_{i+1} a_i^{-1} a_{i+1}^{-2} \text{ for all } i \in \Z/4\Z \,\rangle
\]
(see \cite{Hig51b} or \cite[p.\;18, Prop.\;6]{Ser77b} 
for a discussion of this group).
\smallskip

\textbf{3. } 
In this note the question raised by Lubotzky and van den Dries 
is shown to have an affirmative answer:
\begin{thm}
\label{thm:Answer}
A finitely generated, residually finite group for which the function $\rk-1$ is multiplicative is necessarily free.
\end{thm}

Theorem \ref{thm:Answer} will be a consequence of
\begin{lem}
\label{lem:r-is-primitive}
Let $G\neq\{1\}$ be a finitely generated, residually finite group.
Choose a presentation $\pi_\star \colon F(\XX) /R \iso G$
with $F(\XX)$ a free group of rank $\rk(G)$.
If $R \neq \{1\}$,
select an element $r \in R \smallsetminus \{1\}$ of minimal length.
Then $r$ is a member of a basis of a subgroup $F_1$ of $F$ 
that contains  $R$ and has finite index in $F$.
\end{lem}

\begin{proof}
Let $r$ be represented by the freely reduced word 
$w = s_1s_2 \cdots s_m$ with letters $s_i$ in $\XX \cup \XX^{-1}$.
By the minimal property of $r$,
no non-empty, proper subword $s_i \cdot s_{1+1} \cdots s_j$ of $w$
represents a relator in $R$
and thus the initial segments 
\begin{equation}
\label{eq:Initial-segments}
1, s_1, s_1s_2, \ldots, s_1 s_2 \cdots s_{m-1}
\end{equation}
of $w$ map to distinct elements of $G$.
As the group $G$ is residually finite and finitely generated, 
there exists therefore a \emph{normal} subgroup $G_1\triangleleft{} G$ of finite index 
such that the initial segments listed in \eqref{eq:Initial-segments} represent distinct elements of the quotient group $G/G_1$.
Let $F_1$ denote the full preimage of $G_1$ under the projection 
$\pi \colon F(\XX) \epi G$.  
The property just stated can then be rephrased by saying
that the initial segments \eqref{eq:Initial-segments} of $w$ 
represent distinct cosets of $F_1$ on $F$.
The list \eqref{eq:Initial-segments} 
is thus a partial Schreier transversal of $F_1$ in $F$;
it can be enlarged to a Schreier transversal $\TT_1$ of $F_1$ in $F$.
Let $\XX_1$ of be the basis of $F_1$ obtained from $\TT_1$
by Schreier's method.

Two cases now arise.
If $s_m \in \XX$ the word $w = s_1 \cdots s_{m-1} \cdot s_m$
represents an element of the  basis  $\XX_1$ of $F_1$.
Otherwise, we replace the element $x = s_m^{-1}$ of the basis $\XX$ 
by its inverse $x^{-1} = s_m$ 
and arrive at a new basis 
$\XX' = \left(\XX \smallsetminus \{x\} \right) \cup \{s_m\}$ 
of $F$.
Let $\TT'$ be the set of words obtained from $\TT$ 
by rewriting each element in $\TT$ as a word in $\XX'$.
Then $\TT'$ is a transversal of the subgroup $F_1$ in $F$
and the basis derived from $\XX'$ and $\TT'$ 
by Schreier's method contains the word 
$r^{-1} = s_m^{-1} \cdot (s_{m-1}^{-1} \cdots s_1^{-1})$.
But if so $F_1$ admits again a basis that contains $r$.
\end{proof}

We are left with deducing Theorem  \ref{thm:Answer} 
from Lemma \ref{lem:r-is-primitive}.
Suppose $G$ has rank $m$;
let  $F$ be a free group of the same rank 
and choose an epimorphism $\pi \colon F \epi G$.
If $\pi$ is injective the group $G$ is free, as claimed.
Otherwise, the kernel $R$ of $\pi$ is not reduced to the unit element,
whence Lemma \ref{lem:r-is-primitive} gives us a subgroup $F_1$ 
having finite index in $F$ and comprising $R$, 
and, in addition, a basis $\XX_1$ of $F_1$
that contains an element $r' \in R$.
Since the function $\rk - 1$ is multiplicative on $F$,
this basis has $(m-1) \cdot \card (F/F_1)+1$ elements,
one of which lies in the kernel of $\pi$,
whence $G_1$ is generated by no more than $(\rk G - 1 ) \cdot \card(G/G_1)$ elements.
The function $\rk - 1$ is therefore not multiplicative on $G$,
contrary to hypothesis.
\smallskip

\textbf{4. } 
 One may ask
how far the function $\rho = \rk -1$ can deviate from being multiplicative.
 For some groups an explicit answer is available.
 If, \eg $G$ is free abelian, the function $\rk -1$  is constant;
  if $G$ is an orientable surface group with presentation
  \[
  \langle  a_1, b_1 \ldots, a_g, b_g 
  \mid
    \prod \nolimits_{i = 1}^g [a_i, b_i] = 1 \rangle
  \]
  the facts that subgroups of finite index are again of this form 
  and that the Euler characteristic is multiplicative
  imply the relation
  \begin{equation}
  \label{eq:description}
  \rho(G_1) = [G : G_1] \cdot \rho(G)  +  \left(1 - [G : G_1] \right),
  \end{equation}
  valid for all subgroups $G_1$ of finite index in an orientable surface group.
 \smallskip
  
  \textbf{5. }
  I found the proof of Theorem 1 at the beginning of 1980 
  and communicated my argument to Alex Lubotzky;
  he encouraged me to turn my sketch into a formal proof.
  This task was carried out in 1981;
  the typescript was finished in December of that year.
  I sent one of my preprints to Jean-Pierre Serre;
 he commented on it in a long letter at the beginning of 1982.
 I am very grateful to him for his critical remarks and suggestions.
 They have been incorporated into this new, amended and abridged, 
 version of my original note.
%
%
\bibliography{Text.bbl}

\def\cprime{$'$} \def\cprime{$'$} \def\cprime{$'$} \def\cprime{$'$}
  \def\cprime{$'$} \def\cprime{$'$} \def\cprime{$'$} \def\cprime{$'$}
  \def\cprime{$'$} \def\cprime{$'$} \def\cprime{$'$} \def\cprime{$'$}
  \def\cprime{$'$} \def\cprime{$'$} \def\cprime{$'$} \def\cprime{$'$}
  \def\cprime{$'$} \def\cprime{$'$} \def\cprime{$'$} \def\cprime{$'$}
  \def\cprime{$'$} \def\cprime{$'$} \def\cprime{$'$}
\providecommand{\bysame}{\leavevmode\hbox to3em{\hrulefill}\thinspace}
\providecommand{\MR}{\relax\ifhmode\unskip\space\fi MR }
\providecommand{\MRhref}[2]{%
  \href{http://www.ams.org/mathscinet-getitem?mr=#1}{#2}
}
\providecommand{\href}[2]{#2}
\begin{thebibliography}{LvdD81}

\bibitem[Hig51]{Hig51b}
Graham Higman, \emph{A finitely related group with an isomorphic proper factor
  group}, J. London Math. Soc. \textbf{26} (1951), 59--61. \MR{0038347
  (12,390b)}

\bibitem[LS01]{LS77}
Roger~C. Lyndon and Paul~E. Schupp, \emph{Combinatorial group theory}, Classics
  in Mathematics, Springer-Verlag, Berlin, 2001, Reprint of the 1977 edition.
  \MR{1812024 (2001i:20064)}

\bibitem[LvdD81]{LuDr81}
A.~Lubotzky and L.~van~den Dries, \emph{Subgroups of free profinite groups and
  large subfields of {${\bf Q}$}}, Israel J. Math. \textbf{39} (1981), no.~1-2,
  25--45. \MR{617288}

\bibitem[MKS66]{MKS66}
Wilhelm Magnus, Abraham Karrass, and Donald Solitar, \emph{Combinatorial group
  theory: {P}resentations of groups in terms of generators and relations},
  Interscience Publishers [John Wiley \& Sons, Inc.], New York-London-Sydney,
  1966. \MR{0207802 (34 \#7617)}

\bibitem[Sch27]{Sch27}
Otto Schreier, \emph{Die {U}ntergruppen der freien {G}ruppen}, Abh. Math. Sem.
  Hamburg Univ. \textbf{5} (1927), 161--183.

\bibitem[Ser77]{Ser77b}
Jean-Pierre Serre, \emph{arbres, amalgames, {S}{L}$_2$}, Soci\'et\'e
  Math\'ematique de France, Paris, 1977, avec un sommaire anglais,
  r{\'e}dig{\'e} avec la collaboration de Hyman Bass, Ast{\'e}risque, No. 46.
  \MR{0476875 (57 \#16426)}

\end{thebibliography}
\bibliographystyle{amsalpha}%
\end{document}